\documentclass{article}
\usepackage[utf8]{inputenc}
\usepackage{comment}
\usepackage{amssymb,amsmath,amsthm,graphicx,xcolor,mathtools,enumerate,mathrsfs}
\usepackage{fullpage}
\usepackage{hyperref}
\usepackage{cleveref}

\theoremstyle{plain}
\newtheorem{theorem}{Theorem}[section]
\crefname{theorem}{Theorem}{Theorems}

\crefname{proposition}{Proposition}{Propositions}

\crefname{corollary}{Corollary}{Corollaries}

\newtheorem{lemma}[theorem]{Lemma}
\crefname{lemma}{Lemma}{Lemmas}

\newtheorem{conjecture}[theorem]{Conjecture}
\crefname{conjecture}{Conjecture}{Conjectures}

\crefname{problem}{Problem}{Problem}

\newtheorem{claim}[theorem]{Claim}
\crefname{claim}{Claim}{Claims}

\crefname{observation}{Observation}{Observations}

\crefname{setup}{Setup}{Setups}

\crefname{fact}{Fact}{Facts}

\crefname{algorithm}{Algorithm}{Algorithms}

\crefname{remark}{Remark}{Remarks}

\crefname{example}{Example}{Examples}

\theoremstyle{definition}
\newtheorem{definition}[theorem]{Definition}
\crefname{definition}{Definition}{Definitions}

\crefname{construction}{Construction}{Constructions}

\newtheorem{question}[theorem]{Question}
\crefname{question}{Question}{Questions}

\numberwithin{equation}{section}

\usepackage[shortlabels]{enumitem}
\setlist[enumerate,1]{label={\upshape (\roman*)}}

\renewcommand{\int}[1]{\mathop{\mkern 0mu\mathrm{int}}\nolimits(#1)}

\definecolor{DarkDesaturatedBlue}{HTML}{3A3556}
\definecolor{VividOrange}{HTML}{F15918}
\definecolor{PureOrange}{HTML}{FFBA00}
\definecolor{LightGrayishPink}{HTML}{EEC5D5}
\definecolor{VerySoftBlue}{HTML}{B5AFDB}

\usepackage{tikz}
\usetikzlibrary{math}
\usetikzlibrary{snakes,calc,decorations.pathmorphing,through}
\tikzset{snake it/.style={decorate, decoration=snake}}

\definecolor{DarkDesaturatedBlue}{HTML}{3A3556}
\definecolor{VividOrange}{HTML}{F15918}
\definecolor{PureOrange}{HTML}{FFBA00}
\definecolor{LightGrayishPink}{HTML}{EEC5D5}
\definecolor{VerySoftBlue}{HTML}{B5AFDB}
\begin{document}
\title{Spanning subdivisions in Dirac graphs}
\date{}
\author{Mat\'ias Pavez-Sign\'e\thanks{Supported by the European
Research Council (ERC) under the European Union Horizon 2020 research and innovation programme (grant agreement No. 947978). Mathematics Institute, Zeeman Building, University of Warwick, Coventry CV4 7AL, UK. Matias.pavez-signe@warwick.ac.uk}}

\maketitle
\begin{abstract}
    We show that for every $n\in\mathbb N$ and $\log n\le d< n$, if a graph $G$ has $N=\Theta(dn)$ vertices and minimum degree $(1+o(1))\frac{N}{2}$, then it contains a spanning subdivision of every $n$-vertex $d$-regular graph. \end{abstract}
\section{Introduction}
One of the central questions in extremal graph theory consists of determining degree conditions forcing the containment of rich substructures. One of the most notable examples here is Dirac's theorem~\cite{D1952} from 1952, which states that every graph on $n\ge 3$ vertices and minimum degree at least $\frac{n}{2}$ contains a Hamilton cycle. This result has been highly influential in extremal graph theory for more than 70 years, leading to a large collection of results usually referred to as Dirac-type results. Roughly speaking, the Dirac problem for a family of graphs $\mathcal F$ asks for the \textit{minimum degree threshold} which forces the containment of all members of $\mathcal F$.
\begin{definition}\label{definition:threshold}For an $n$-vertex host graph $G$ and a family $\mathcal F$ of graphs on at most $n$ vertices, the minimum degree threshold for containing $\mathcal F$ is
\[\delta(n,\mathcal F)=\min\{m:\delta(G)\ge m\mbox{ implies }F\subseteq G\mbox{ for every }F\in\mathcal F\}.\]
\end{definition}
Due to the emergence of sophisticated embedding techniques in the last few decades, the Dirac problem is nowadays well-understood for a large class of graphs. For instance, the family of $n$-vertex trees with maximum degree $o(n/\log n)$  has minimum degree threshold $(1+o(1))\frac{n}{2}$ (see~\cite{kathapurkar2022spanning,KSS1995}), as conjectured by Bollob\'as~\cite{bollobas2004extremal} in the late 70s. For $n$ divisible by $d$, the threshold for containing a $K_d$-factor is $(1-1/d)n$ as shown by Hajnal and Szemer\'edi~\cite{HS1970}. For a general graph $H$, the minimum degree threshold for the containment of a \textit{perfect $H$-tiling} was determined (up to a constant additive term) by K\"uhn and Osthus~\cite{kuhn2009minimum} (see~\cite{Komlos2000TilingTT,Shokoufandeh2003ProofOA} for the threshold of \textit{almost perfect tilings} and~\cite{hurley2022sufficient} for a recent result on \textit{mixed tilings}).  The threshold for powers of Hamilton cycles was determined by Koml\'os, S\'ark\"ozy, and Szemer\'edi~\cite{KSS1998,KSS1998b}, proving the celebrated P\'osa--Seymour conjecture for large graphs (see~\cite{seymour73} for example). A breakthrough in the area was the proof of the Bollob\'as--Koml\'os bandwidth conjecture~\cite{komlos_1999} by B\"ottcher, Schacht, and Taraz~\cite{BST2009}, which asymptotically establishes the minimum degree threshold for the class of bounded degree $n$-vertex graphs with bounded chromatic number and sublinear bandwidth. For a more extensive revision of Dirac-type problems, we recommend looking at the excellent surveys~\cite{kuhn2009embedding,simonovits2019embedding}.

In this paper, we investigate Dirac-type problems for subdivisions of graphs.  Given a graph $H$, say a graph $H'$ is an \textit{$H$-subdivision} if $H'$ is obtained from $H$ by replacing one or more edges from $H$ with vertex-disjoint paths. Very recently, Lee~\cite{lee2023perfect} determined the minimum degree threshold for the containment of a \textit{perfect $H$-subdivision tiling}, that is, a collection of vertex-disjoint subdivisions of $H$ covering every vertex in the host graph. In this article, we will continue this study by investigating the minimum degree threshold of certain spanning graph subdivisions.

Problems regarding degree conditions forcing the containment of subdivisions have been extensively studied since the 60s, when Mader~\cite{MADER1967} proved that graphs with large average degree contain clique subdivisions. Independently, Mader~\cite{MADER1967} and Erd\H os and Hajnal~\cite{erdos1964complete} conjectured that graphs with average degree $\Omega(k^2)$ contain a subdivision of the clique on $k$ vertices, which is optimal as the complete bipartite graph with parts of size $\frac{k^2}{16}$ cannot contain a subdivision of $K_k$, the complete graph on $k$ vertices. This conjecture was solved in the 90s by Bollob\'as and Thomason~\cite{BOLLOBAS1998883} and by Koml\'os and Szemer\'edi~\cite{KomlosSzemeredi_subdivisions}. Two variations of this conjecture have been proposed:
\begin{itemize}
    \item Thomassen's conjecture~\cite{Thomassen1984}: For each $k\in\mathbb N$, there exists some $d=d(k)$ such that every graph with average degree at least $d$ contains a \textit{balanced}\footnote{A \textit{balanced} subdivision of a graph $H$ is an $H$-subdivision where all the edges of $H$ are replaced with paths of the same length.} subdivision of $K_k$.
    \item Verstra\"ete's conjecture~\cite{verstraete_2002}: If $G$ is a graph with average degree at least $\Omega(k^2)$, then $G$ contains a pair of disjoint isomorphic subdivisions of $K_k$.
\end{itemize}
These two conjectures have been recently solved. Firstly, Thomassen's conjecture was solved by Liu and Montgomery~\cite{liu2023solution} in their recent solution of the Erd\H os--Hajnal odd cycle problem~\cite{erdoshajnal66}. Secondly, Luan, Tang, Wang, and Yang~\cite{luan2023balanced} and Gil Fern\'andez, Hyde, Liu, Pikhurko, and Wu~\cite{fernandez2023disjoint} showed that graphs with average degree $\Omega(k^2)$ contain two pairwise disjoint isomorphic balanced subdivisions of $K_k$, settling Verstra\"ete's conjecture and also giving optimal bounds for Thomassen's conjecture. Note that for dense graphs, the above-mentioned results imply that every $n$-vertex graph with $\Omega(n^2)$ edges contains a balanced subdivision of a clique on $\Omega(\sqrt{n})$ vertices. Answering an old question of Erd\H os~\cite{erdostop}, Alon, Krivelevich, and Sudakov~\cite{alon_krivelevich_sudakov_2003} (see also~\cite{fox2011dependent}) showed that every $n$-vertex graph with at least $\varepsilon n^2$ edges contains a \textit{$1$-subdivision}\footnote{For $\ell\in\mathbb N$, an \textit{$\ell$-subdivision} of a graph $H$ is an $H$-subdivision where each edge of $H$ is replaced with a path with exactly $\ell$ interior vertices.} of the clique on $\varepsilon\sqrt{n}/4$ vertices. In particular, if an $n$-vertex graph $G$ has minimum degree $\delta(G)\ge\frac{n}{2}$, then it has at least $\frac{1}{4}n^2$ edges and thus contains a $1$-subdivision of the clique of size $\sqrt{n}/16$. 

Our first result is that we can embed a spanning subdivision of a clique in graphs with minimum degree slightly above Dirac's condition. 
\begin{theorem}\label{theorem:main2}For every $\varepsilon>0$, there exists a positive constant $C_0$ such that for all $C\ge C_0$ and $n\ge 2$ the following holds. Let $G$ be a graph on $N=C(n-1)n$ vertices and minimum degree $\delta(G)\ge (1+\varepsilon)\frac{N}{2}$. Then, $G$ contains a spanning subdivision of $K_n$.
\end{theorem}
The constant $\frac{1}{2}$ in the minimum degree condition in Theorem~\ref{theorem:main2} is essentially best possible, as shown by considering the graph consisting of two disjoint cliques of size $C\binom{n}{2}$. Moreover, Theorem~\ref{theorem:main2} is a direct consequence of the following more general result about subdivisions of $d$-regular graphs.
   
\begin{theorem}\label{theorem:main}For every $\varepsilon>0$, there exists a positive constant $C_0$ such that for all $C\ge C_0$ and $n\ge 2$ the following holds. Let $\log n\le d< n$ and let $G$ be a graph on $N=Cdn$ vertices and minimum degree $\delta(G)\ge (1+\varepsilon)\frac{N}{2}$. Then, $G$ contains a spanning subdivision of every $n$-vertex $d$-regular graph.
\end{theorem}
The proof of Theorem~\ref{theorem:main} is probabilistic in nature and avoids any use of the \textit{regularity lemma}. In particular, we can pick the constant $C_0$ in Theorem~\ref{theorem:main} to be of order $C_0=O(\varepsilon^{-4})$ instead of being of tower-type. Also, from the proof of Theorem~\ref{theorem:main}, we can embed a balanced subdivision covering all but a linear proportion of the vertices of the host graph. Thus, by randomly partitioning the vertex set into $k\in\mathbb N$ subsets of roughly the same size, one can find $k$ vertex-disjoint isomorphic balanced subdivisions covering all but a linear proportion of the vertices. 

\section{Proof}
\subsection{Notation}
 For a graph $G$, we let  $V(G)$ and $E(G)$ denote its vertex set and edge set, respectively, and we write $|G|=|V(G)|$. The neighbourhood of a vertex $x$, denoted $N(x)$, is the set of vertices that are adjacent to $x$, and we write $d(x)=|N(x)|$ for the degree of $x$. The minimum degree of $G$ is denoted $\delta(G)$. Also, for a vertex $x\in V(G)$ and a subset $U\subset V(G)$, we write $N(x,U)$ for the set of neighbours of $x$ in $U$ and let $d(x,U)=|N(x,U)|$ denote the degree of $x$ into $U$. Given a set $A\subset V(G)$, we denote by $G[A]$ the graph induced by $A$, and, for disjoint subsets $A,B\subset V(G)$, we let $G[A,B]$ denote the bipartite graph induced by $A$ and $B$. When working with more than one graph, we use subscripts to explicit which graph are we working with, for instance, $d_H(x)$ denotes the degree of a vertex $x$ in the graph $H$.

A path $P$ is an ordered sequence of distinct vertices $P=u_1\ldots u_{t+1}$ such that $u_iu_{i+1}$ is an edge for each $i\in[t]$, in which case we say that $P$ has length $t$, endpoints $u_1$ and $u_{t+1}$, and interior $u_2,\ldots, u_t$. For vertices $x$ and $y$, an $x,y$-path is a path with endpoints $x$ and $y$. For a path $P=v_1\ldots v_t$, we let $P^\triangleleft=v_t\ldots v_1$ denote the path $P$ traversed in reverse order. For an $H$-subdivision $H'$, we say that a vertex $v\in V(H')$ is a \textit{branch vertex} if it is the copy of some vertex from $H$. 

 As usual, $[n]$ denotes the set of the first $n$ positive integers. Also, we use standard hierarchy notation, that is, we write $a\ll b$ to denote that given $b$ one can choose $a$ sufficiently small so that all relevant statements hold.
\subsection{Probabilistic tools}
Let $N,m,n\in\mathbb N$ satisfy $m,n\le N$, let $J$ be a set of size $N$ and let $I\subset J$ be a subset of size $m$. If a subset $I'\subset J$ is chosen uniformly at random amongst all subsets of size $n$, then the random variable $X=|I\cap I'|$ is said to have a \textit{hypergeometric distribution} with parameters $N$, $n$, and $m$. For example, for a graph $G$ and a vertex $v\in V(G)$, if we pick a random subset $U\subset V(G)$ of size $\ell$, then $d(v,U)$ is a hypergeometric random variable with parameters $N'=|G|$, $n'=\ell$ and $m'=d(v)$. We will use the following standard concentration result for hypergeometric random variables. 

\begin{lemma}[Theorem 2.10 in~\cite{JLR2000}]\label{lemma:chernoff}Let $X$ be a hypergeometric random variable with parameters $n,m$, and $N$. Then, for every $t>0$, 
\[\mathbb P\left(\big|X-\mathbb E[X]\big|\ge t\right)\le 2e^{-2t^2/n}.\]
\end{lemma}
\begin{definition}Let $H$ be a graph with vertex set $V(H)=[t]$. An $(H,\alpha)$-good partition of a graph $G$ is a vertex partition $V(G)=V_1\cup\ldots\cup V_t$ which satisfies the following properties.
\begin{enumerate}
    \item $m=|V_1|=\dots=|V_t|$.
    \item$\delta(G[V_i])\ge \alpha m$ for all $i\in [t]$.
    \item $\delta(G[V_i,V_j])\ge \alpha m$ for all $i,j\in [t]$ such that $ij\in E(H)$.
\end{enumerate}
\end{definition}
We now show that a randomly chosen partition is likely to be good when the host graph has linear minimum degree.
\begin{lemma}\label{lemma:partition:1}
    Let $1/C\ll \delta<\alpha$ and $n\ge 2$. Suppose $\log n\le d< n$ and that $G$ is a graph on $N=Cdn$ vertices and minimum degree $\delta(G)\ge \alpha N$. Then,  for every $d$-regular graph $H$ on $n$ vertices, $G$ admits an $(H,\alpha-\delta)$-good partition.
    \end{lemma}
\begin{proof}Let $V(G)=V_1\cup\ldots\cup V_n$ be a random partition of $V(G)$ into sets of size $Cd$. Then, for any vertex $v\in V(G)$ and index $i\in [n]$, $d(v,V_i)$ follows a hypergeometric distribution with parameters $N'=N$, $n'=|V_i|$ and $m'=d(v)$, and has expectation
\[\mathbb E[d(v,V_i)]=d(v)\cdot \frac{|V_i|}{N}\ge \alpha Cd.\]
Then, from Lemma~\ref{lemma:chernoff}, we have
\[\mathbb P(d(v,V_i)\le (\alpha-\delta)Cd)\le 2e^{-2\delta^2Cd}.\]
Therefore, using the union bound and that $H$ is $d$-regular, we have
\[\mathbb P((V_1,\ldots,V_n)\mbox{ is not }(H,\alpha-\delta)\mbox{-good})\le 2 |G|\cdot d\cdot 4e^{-2\delta^2Cd}\le 2Cn^3e^{-2\delta^2C\log n}<1.\]
\end{proof}

\begin{lemma}\label{lemma:partition:2}Let $1/C\ll\delta<\alpha$ and $d\in\mathbb N$, and suppose $G$ is a graph on $N=Cd$ vertices and minimum degree $\delta(G)\ge \alpha N$. Then, for every set of distinct vertices $v_0,v_1,\ldots, v_d\in V(G)$, there is a partition $V(G)\setminus\{v_0,v_1,\ldots, v_d\}=V_1\cup\ldots\cup V_d$ such that 
\begin{enumerate}
 \item\label{partition:degree:1} $|V_1|=\dots=|V_{d-1}|=C-1$ and $|V_d|=C-2$, 
 \item\label{partition:degree:2} $d(v_0,V_i)\ge (\alpha-\delta)|V_i|$ and $d(v_i,V_i)\ge (\alpha-\delta)|V_i|$ for all $i\in [n]$, and 
 \item\label{partition:degree:3} $\delta(G[V_i])\ge (\alpha-\delta)C$ for all $i\in [n]$.
\end{enumerate}
\end{lemma}
\begin{proof}Set $m_1=\dots=m_{d-1}=C-1$ and $m_d=C-2$. We start by dividing $[d]$ into a collection of subintervals that we will use as a guide for constructing our partition. Let $I_{0,1}=[d]$ and let $s\ge 0$ be minimal integer such that $2^{s-1}< d\le 2^{s}$. Iteratively, for $0\le i<s$ and $1\le j\le 2^{i}$, we partition $I_{i,j}$ into intervals $I_{i+1,2j-1}$ and $I_{i+1,2j}$ such that $\big||I_{i+1,2j-1}|-|I_{i+1,2j}|\big|\le 1$.
Note that at each step $I_{i,j}$ is divided into two subintervals of length roughly $|I_{i,j}|/2$ and thus, after $i$ steps, $[d]$ is divided into $2^{i}$ subintervals of length approximately $d/2^{i}$. Furthermore, as $\big||I_{i+1,2j-1}|-|I_{i+1,2j}|\big|\le 1$ for all $0\le i<s$ and $1\le j\le 2^i$, by definition of $s$ each $I_{s-1,\ell}$, for $\ell\in [2^{s-1}]$, is either a singleton or contains exactly $2$ elements.

Set $V_{0,1}=V(G)\setminus\{v_0,v_1,\ldots, v_d\}$. Iteratively, for $0\le i<s$ and $1\le j\le 2^i$, partition $V_{i,j}$ uniformly at random into sets $V_{i+1,2j-1}$ and $V_{i+1,2j}$ such that 
\[|V_{i+1,2j-1}|=\sum_{\ell\in I_{i+1,2j-1}}m_\ell\quad \text{and}\quad |V_{i+1,2j}|=\sum_{\ell'\in I_{i+1,2j}}m_{\ell'}.\]
This is possible as $I_{i,j}=I_{i+1,2j-1}\cup I_{i+1,2j}$ and $|V_{i,j}|=\sum_{\ell\in I_{i,j}}m_\ell$. Let $J\subset [2^{s-1}]$ be the set of those indices $j\in [2^{s-1}]$ such that $I_{s-1,j}$ is a singleton. Let $t=|J|$ and note that $0\le t<2^{s-1}$, as $2^{s-1}<d\le 2^s$. After relabelling the indices, we may assume that $[2^{s-1}]\setminus J=[2^{s-1}-t]$. At step $s$, for each $1\le j\le 2^{s-1}-t$ we partition $V_{s-1,j}$ uniformly at random into sets $V_{s,2j-1}$ and $V_{s,2j}$ such that $\big||V_{s,2j-1}|-|V_{s,2j}|\big|\le 1$, and, noting that $d=2^s-t$, for each $2^{s-1}-t<j'\le 2^{s-1}$ we set $V_{s,2^{s-1}-t+j'}=V_{s-1,j'}$. For $0\le i<s$ and $1\le j\le 2^i$, let $\mathbf{Y}_{i,j}$ denote the event where the following three properties hold: 
\begin{enumerate}[label = \upshape\textbf{A\arabic{enumi}}]
\item\label{degree:1} $\delta(G[V_{i,j}])\ge (\alpha-2|V_{i,j}|^{-1/4})|V_{i,j}|$. 
    \item\label{degree:2} $d(v_0,V_{i,j})\ge (\alpha-2|V_{i,j}|^{-1/4})|V_{i,j}|$.
    \item\label{degree:3} $d(v_\ell, V_{i,j})\ge (\alpha-2|V_{i,j}|^{-1/4})|V_{i,j}|$ for all $\ell\in I_{i,j}$.
\end{enumerate}
For $j\in [d]$, let $\mathbf{Y}_{s,j}$ denote the event where \ref{degree:1}, \ref{degree:2} and the following hold: 
\begin{enumerate} [label=\upshape\textbf{A4}]
    \item $d(v_j,V_{s,j})\ge (\alpha-2|V_{s,j}|^{-1/4})|V_{s,j}|$.\label{degree:4}
\end{enumerate}
For $0\le i\le s$, let $\mathbf{Y}_{i}$ be the event where $\mathbf{Y}_{i,j}$ holds for all $1\le j\le \min\{2^{i},d\}$. 
\begin{claim}\label{claim:chernoff}For all $1\le i\le s$, $\mathbb P(\mathbf{Y}_i\mid \mathbf{Y}_{i-1})\ge \exp(-2^{i-1})$.    
\end{claim}
Before proving Claim~\ref{claim:chernoff}, let us show how to conclude the proof. We observe that if $\mathbf{Y}_s$ holds, then we would have found the required partition. Indeed, after a relabelling of the indices, each set $V_{s,j}$ has size exactly $m_j$ and \ref{degree:1}--\ref{degree:4} imply that \ref{partition:degree:1}--\ref{partition:degree:3} hold, as $1/m_j\ll \delta$.  Therefore, it is enough to show that $\mathbf{Y}_s$ occurs with positive probability. Noting
that $\mathbf{Y}_0$ holds with probability 1, as $\delta(G)\ge \alpha N$, from Claim~\ref{claim:chernoff} we have 
\[\mathbb P(\mathbf{Y}_s)\ge \prod_{i=1}^s\mathbb P(\mathbf Y_i\mid \mathbf{Y}_{i-1})\ge \exp\left(-\sum_{i=1}^s2^{i-1}\right)\ge e^{-2d},\]
which finishes the proof. It is left only to prove Claim~\ref{claim:chernoff}.

\begin{proof}[Proof of Claim~\ref{claim:chernoff}]Let $1\le i<s$ and suppose that we have chosen our partition $V(G)=\bigcup_{j\in[2^{i-1}]}V_{i-1,j}$ for which $\mathbf{Y}_{i-1}$ holds. In what follows, all probabilities and expectations are conditioned on $\mathbf{Y}_{i-1}$. Let $1\le j\le 2^{i-1}$. For a vertex $v\in V_{i-1,j}$, observe that $d(v,V_{i,2j})$ follows a hypergeometric distribution with parameters $N'=|V_{i-1,j}|$, $n'=|V_{i,2j}|$ and $m'=d(v,V_{i-1,j})$. Then, assuming $\mathbf{Y}_{i-1}$ and hence $\mathbf{Y}_{i-1,j}$, from~\ref{degree:1} we deduce that
\[\mathbb E[d(v,V_{i,2j})]\ge (\alpha-2|V_{i-1,j}|^{-1/4})|V_{i,2j}|\]
and thus, from Lemma~\ref{lemma:chernoff}, 
\begin{equation}\label{chernoff:main:lemma}\mathbb P\left(d(v,V_{i,2j})\le (\alpha -2|V_{i-1,j}|^{-1/4})|V_{i,2j}|-|V_{i,2j}|^{2/3}\right)\le 2 \exp(-2|V_{i,2j}|^{1/3}).\end{equation}
On the other hand, using that $|V_{i,2j}|\le \frac{11}{20}|V_{i-1,j}|$, we have
    \begin{equation}\label{eq:main:lemma}2|V_{i-1,j}|^{-1/4}\cdot |V_{i,2j}|+|V_{i,2j}|^{2/3}\le
  2\cdot \left(\tfrac{11}{20}\right)^{1/4}|V_{i,2j}|^{3/4}+\tfrac{1}{10}|V_{i,2j}|^{3/4}\le 2|V_{i,2j}|^{3/4},\end{equation}
where we have used that $|V_{i,2j}|\ge C-2$ and $1/C\ll 1$. Combining \eqref{chernoff:main:lemma} and \eqref{eq:main:lemma}, we finally have
\[\mathbb P\left(d(v,V_{i,2j})\le (\alpha-2|V_{i,2j}|^{-1/4})|V_{i,2j}|\right)\le 2\exp(-2|V_{i,2j}|^{1/3}).
\]
Using the union bound over all $v\in V_{i-1,j}$, it follows that~\ref{degree:1} holds for $i$ and $2j$ with probability at least 
\begin{equation}\label{chernoff:2main:lemma}1-2|V_{i-1,j}|\cdot\exp(-2|V_{i,2j}|^{1/3})\ge 1-2|V_{i-1,j}|\cdot\exp(-|V_{i-1,j}|^{1/3}),\end{equation}
where we have used that $2|V_{i,2j}|^{1/3}\ge 2\cdot \left(\tfrac{9}{20}\right)^{1/3}|V_{i-1,j}|^{1/3}\ge|V_{i-1,j}|^{1/3}$. Similar calculations show that \ref{degree:2} and \ref{degree:3} hold with probability at least~\eqref{chernoff:2main:lemma}, and clearly the same conclusions hold for $2j-1$. Therefore, the probability that $\mathbf{Y}_{i,2j-1}$ and $\mathbf{Y}_{i,2j}$ hold, given a partition $V(G)=\bigcup_{j\in [2^{i-1}]}V_{i-1,j}$ for which $\mathbf{Y}_{i-1,j}$ holds, is at least 
\begin{equation}\label{prob:1}1-12|V_{i-1,j}|\cdot\exp(-|V_{i-1,j}|^{1/3}).\end{equation}
Finally, using that the events $(\mathbf{Y}_{i,2j-1}\wedge \mathbf{Y}_{i,2j})_{1\le j\le 2^{i-1}}$ are mutually independent given a partition $V(G)=V_{i-1,1}\cup\ldots\cup V_{i-1,2^{i-1}}$ for which $\mathbf{Y}_{i-1}$ holds, we can compute
\begin{eqnarray*}\mathbb P(\mathbf{Y}_i\mid \mathbf{Y}_{i-1})\overset{\eqref{prob:1}}{\ge} \prod_{j=1}^{2^{i-1}}(1-12 |V_{i-1,j}|\cdot \exp(-|V_{i-1,j}|^{1/3})&\ge& \prod_{j=1}^{2^{i-1}}\exp(-24|V_{i-1,j}|\cdot\exp(-|V_{i-1,j}|^{1/3}))\\
&\ge & \exp(-2^{i-1}),\end{eqnarray*}
where we have used that $e^{-2x}\le 1-x$ holds for all $0<x\le \frac{1}{2}$ and that $|V_{i-1,j}|\cdot\exp(-|V_{i-1,j}|^{1/3})\ll 1$ as $|V_{i-1,j}|\ge C-2$ and $1/C\ll 1$. Similar arguments show that $\mathbb{P}(\mathbf{Y}_s\mid \mathbf{Y}_{s-1})\ge \exp(-2^{s-1})$. 

\end{proof}
\end{proof}
\subsection{Proof of Theorem~\ref{theorem:main}}
 We say that a graph $G$ is \textit{Hamiltonian path-connected} if for every pair of distinct vertices $x,y\in V(G)$ there is a Hamiltonian $x,y$-path in $G$. The last ingredient that we need is the following classical result due to Ore~\cite{ore1963hamilton}.
\begin{lemma}\label{lemma:hamilton}Every graph $G$ on $n\ge 4$ vertices and minimum degree $\delta(G)\ge \frac{n+1}{2}$ is Hamiltonian path-connected.
\end{lemma}
\begin{proof}[Proofof Theorem~\ref{theorem:main}]We start by picking a constant $1/C_0\ll\varepsilon$ and let $C\ge C_0$ and $n\ge 2$. Let $\log n\le d<n$ and let $H$ be a $d$-regular graph with vertex set $V(H)=[n]$. Given a graph $G$ with $N=Cdn$ vertices and minimum degree $\delta(G)\ge (1+\varepsilon)N/2$, we now find a partition of $V(G)$ which will be used as a template to embed the $H$-subdivision. 
\begin{claim}\label{claim:partition}
	There are distinct vertices $v_1,\ldots, v_n$ and pairs $(V_{i,j},v_{i,j})$, $i\in [n]$ and $j\in N_H(i)$, where $V_{i,j}\subset V(G)$ is a subset and $v_i,v_{i,j}\in V_{i,j}$, such that the following properties hold. 
\end{claim}
\begin{enumerate}[label = \upshape\textbf{B\arabic{enumi}}]
	\item\label{B:1} $V(G)=\bigcup_{i\in [n], j\in N_H(i)}V_{i,j}$.
 \item\label{B:2} $V_{i,j}$ and $V_{k,\ell}$ are disjoint for all $i\not= k$ and $j\in N_H(i), \ell\in N_H(k)$.
 \item\label{B:3} $V_{i,j}\cap V_{i,k}=\{v_i\}$ for every $i\in [n]$ and distinct $j,k\in N_H(i)$.
\item\label{B:4} $|V_{i,j}|\in\{C,C+1\}$ for all $i\in [n]$ and $j\in N_H(i)$.
    \item\label{B:5} $\delta(G[V_{i,j}])\ge\frac{|V_{i,j}|+1}{2}$ for all $i\in [n]$ and $j\in N_H(i)$.
    \item\label{B:6} For every $i\in [n]$ and $j\in N_H(i)$, $v_{ij}v_{ji}\in E(G)$.
\end{enumerate}
For now, let us assume Claim~\ref{claim:partition} in order to show Theorem~\ref{theorem:main}. Firstly, we use the vertices $(v_i)_{i\in [n]}$ as the branch vertices of the $H$-subdivision. Secondly, for each $i\in [n]$ and $j\in N_H(i)$, find a Hamilton $v_i,v_{i,j}$-path $P_{i,j}$ inside $V_{i,j}$. This is possible since \ref{B:5} and Lemma~\ref{lemma:hamilton} imply that $G[V_{i,j}]$ is Hamiltonian path-connected. We claim that this gives the spanning $H$-subdivision. Indeed, for all $ij\in E(H)$, \ref{B:6} implies that $Q_{ij}=P_{i,j}P^\triangleleft_{j,i}$ is a $v_iv_j$-path (see Figure~\ref{fig:connectivity}) which, because of~\ref{B:2} and~\ref{B:3}, is internally disjoint from every other possible path $Q_{k\ell}$. This gives an $H$-subdivision $F\subset G$. Finally, \ref{B:1} clearly implies that $F$ uses all vertices from $G$. 

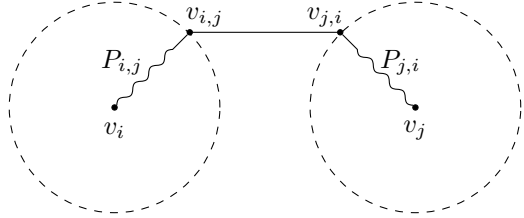
\begin{figure}[ht]
    \centering
      \begin{tikzpicture}[scale=1]

\draw node at (-.8,2.2) {$v_{i,j}$};
\draw node at (-2,.7) {$v_i$};
\draw node at (2,.7) {$v_j$};
\draw node at (.8,2.2) {$v_{j,i}$};
\draw node at (-1.9,1.6) {$P_{i,j}$};
\draw node at (1.8,1.6) {$P_{j,i}$};

\draw[dashed] (-2,1) circle [radius=1.4];
\draw[dashed] (2,1) circle [radius=1.4];

\tikzstyle{every node}=[circle, draw, fill, inner sep=0pt, minimum width=2pt]
\draw node (a1) at (-1,2) {};
\draw node (b4) at (1,2) {};
\draw node (x) at (-2,1) {};
\draw node (y) at (2,1) {};

\draw (a1) to (b4);
\draw[snake=coil,segment aspect=0, segment amplitude=.4mm,segment length=3mm] (x) -- (a1);
\draw[snake=coil,segment aspect=0, segment amplitude=.4mm,segment length=3mm] (y) -- (b4);

\end{tikzpicture}
 \caption{Creating the $v_i,v_j$-path $Q_{ij}$.}
 \label{fig:connectivity}
\end{figure}

\begin{proof}[Proof of  Claim~\ref{claim:partition}]
	Use Lemma~\ref{lemma:partition:1}, with parameter $\delta=\varepsilon/2$, to find an $(H,\frac{1}{2}+\frac{\varepsilon}{2})$-good partition $V(G)=\bigcup_{i\in [n]}V_{i}$ such that 
	\begin{enumerate}
		\item$|V_{i}|=Cd$ for all $i\in [n]$,
		\item$\delta(G[V_i])\ge (1+\tfrac{\varepsilon}{2})\frac{Cd}{2}$, and
		\item$\delta(G[V_i,V_j])\ge (1+\tfrac{\varepsilon}{2})\frac{Cd}{2}$ for all $ij\in E(H)$.
	\end{enumerate}
Then, for each $ij\in E(H)$ pick vertices $v_{i,j}\in V_i$ and $v_{j,i}\in V_j$ such that $v_{i,j}v_{j,i}\in E(G)$, and, for $i\in [n]$, pick an arbitrary vertex $v_i\in V_i$, choosing new vertices at each time. For each $i\in [n]$, use Lemma~\ref{lemma:partition:2}, with parameter $\delta=\varepsilon/4$, to find a partition $V_i=\{v_i\}\cup\{v_{i,j}:j\in N_H(i)\}\cup \bigcup_{j\in N_H(i)}V_{i,j}'$ such that 
\begin{itemize}
    \item $|V'_{i,j}|\in\{C-2,C-1\}$ for all $j\in N_H(i)$,
    \item $d(v_{i,j},V'_{i,j}), d(v_i,V'_{i,j})\ge (1+\tfrac{\varepsilon}{4})\frac{|V'_{i,j}|}{2}$, and
    \item $\delta(G[V'_{i,j}])\ge (1+\tfrac{\varepsilon}{4})\frac{|V'_{i,j}|}{2}$.
\end{itemize}
 Finally, for each $i\in [n]$ and $j\in N_H(i)$, we set $V_{i,j}=V_{i,j}'\cup \{v_i,v_{i,j}\}$, which clearly satisfies \ref{B:1}--\ref{B:6}.
\end{proof}\end{proof}
\section{Concluding remarks}Let us observe that having a host graph on $N=Cdn$ vertices in Theorem~\ref{theorem:main} is just an artificial setup to make the subdivision as balanced as possible. However, by adjusting the length of the connecting paths, the same arguments work if we put $N=\Theta(dn)$ or if we drop the condition of being $d$-regular, that is, one can show that if the host graph $G$ has $N=\Theta(m)$ vertices, where $m\le \binom{n}{2}$, and $H$ is an $n$-vertex graph with $m$ edges and minimum degree at least $\log n$, then $G$ contains a spanning $H$-subdivision. Having said this, it is tentative to conjecture that the condition $d\ge \log n$ in Theorem~\ref{theorem:main} could be totally dropped.

\begin{conjecture}\label{conjecture:1}For every $\varepsilon>0$, there exists a positive constant $C_0$ such that for all $C\ge C_0$ and $m\in\mathbb N$ the following holds.  Let $G$ be a graph on $N=Cm$ vertices and minimum degree $\delta(G)\ge (1+\varepsilon)\frac{N}{2}$. Then, $G$ contains a spanning $H$-subdivision of every graph $H$ with $m$ edges and no isolated vertices.
\end{conjecture}
After this article was submitted, Conjecture~\ref{conjecture:1} was solved by Lee~\cite{lee2023spanning} using the \textit{absorption method}. The only drawback in Lee's proof is that it requires one of the paths in the subdivision to be very long compared with the other paths. So, it would be still interesting to find a spanning subdivision where all the paths in the subdivision have nearly the same length. 
\begin{question} Is it possible to find a spanning $H$-subdivision in the context of Conjecture~\ref{conjecture:1} where all the paths in the subdivision have similar lengths?
\end{question}
Let us finish by mentioning that similar arguments as those in the proof of Theorem~\ref{theorem:main} can be used to find a perfect $H$-subdivision tiling, where $H$ is a graph whose size could even grow depending on the size of the host graph. However, a more precise result was shown by Lee~\cite{lee2023perfect} in this context. Indeed, the leading constant in the threshold for containing a perfect $H$-subdivision tiling can be smaller than $\frac{1}{2}$ in certain cases.

\bibliographystyle{abbrv}
\bibliography{subdivision-dirac.bib}

\end{document}